\documentclass[
final
]{dmtcs-episciences}

\usepackage[utf8]{inputenc}
\usepackage{subfigure}
\usepackage{amsmath}

\usepackage[square]{natbib}
\setcitestyle{numbers}

\newtheorem{theorem}{Theorem}[section]
\newtheorem{thm}[theorem]{Theorem}
\newtheorem{lemma}{Lemma}[section]

\newtheorem{corollary}[lemma]{Corollary}

\newtheorem{remark}{Remark}[section]
\newtheorem{definition}{Definition}[section]

\author{Hui Rao\affiliationmark{1}\thanks{The work is supported by NSFS Nos. 11431007 and 11971195.}
  \and Lei Ren\affiliationmark{1,2}\thanks{The corresponding author.}
  \and Yang Wang\affiliationmark{3} \thanks{The work is supported by HK GRF 16317416.}
  }
\title[Formatting an article for DMTCS]{Dissecting a square into congruent polygons}
\affiliation{
  School of Mathematics and  Statistics, CCNU, China\\
  School of Mathematics and Statistics,   XYNU, China\\
  Department of Mathematics, HKUST, Hong Kong}
\keywords{ Tiling, Eulerian graph, hypotenuse graph}
\received{2020-1-13}
\revised{2020-5-13}
\accepted{2020-5-25}
\begin{document}
\publicationdetails{22}{2020}{1}{21}{6022}
\maketitle
\begin{abstract}
  We study the dissection of a square into congruent convex polygons.
Yuan \emph{et al.} [Dissecting the square into five congruent parts, Discrete Math. \textbf{339} (2016) 288-298]
asked whether, if the number of tiles is a prime number $\geq 3$,
it is true that  the tile must be a rectangle.
 We conjecture that the same conclusion still holds even if
  the number of tiles is an odd number $\geq 3$.
  Our conjecture has been confirmed  for triangles in earlier works.
We prove that  the conjecture holds if either the tile is a convex $q$-gon with $q\geq 6$ or
 it is a right-angle trapezoid.
\end{abstract}

\section{Introduction}

Let $\Omega$ be a polygon in $\reals^2$, and let
  $\{P_j;~j=1,\dots, N\}$ be a  family of
polygons. We call $\{P_j\}_{j=1}^N$ a \emph{tiling} or \emph{dissection} of $\Omega$,
if
\begin{math}
\Omega=\bigcup_{j=1}^N  P_j
\end{math}
and the right hand side is a non-overlapping union, that is, the interiors of the tiles are pairwise disjoint.
In particular, we are interested in the tiling
\begin{equation}\label{eq:Danzer}
\Omega=\bigcup_{j=1}^N  P_j,
\end{equation}
where $\Omega$ is a square, and all $P_j$, $j\in \{1,\dots, N\}$, are congruent to a convex polygon $P$.
In this case, we also say that $P$ can tile $\Omega$. (Two sets $A$ and $B$ are congruent if $A=g(B)$ where $g$
is a composition of a rotation, possibly a reflection and a translation.)

In the 1980's, Ludwig Danzer conjectured   that
if  $N=5$ in (\ref{eq:Danzer}), then $P$ must be a rectangle (see \citep{Yua2016}).
Yuan \emph{et al.} \cite{Yua2016} proved that Danzer's conjecture is true, and
asked whether, if the number of tiles is a prime number $\geq 3$,
it is true that  the tile must be a rectangle.
Except $N=5$, this question was answered confirmatively for $N=3$ in \cite{Mal1994}.

In this paper, we formulate a stronger conjecture:

\medskip
\noindent \textbf{Conjecture 1.} \emph{If  a convex polygon $P$ can tile a square  and the number of
tiles is an odd number  $\geq 3$, then $P$ must be a rectangle.}
\medskip

A polygon is called a \emph{q-gon} if it has $q$ vertices and thus $q$ sides.
Instead of considering the number of tiles $N$ case by case as in \cite{Mal1994, Yua2016}, we study the problem on $q$ case by case.

When $q=3$, Conjecture 1 is confirmed by Thomas \cite{Tho1968} and Monsky \cite{Mon1970}. Actually, they proved the following surprising  result: If a rectangle is tiled by $N$ triangles with the same area, then $N$ must be an even number.

Our first result is to  show that Conjecture 1 is true for $q\geq 6$. Actually, we prove the following stronger result:

\begin{theorem}\label{thm:six-gon} Let $R$ be a rectangle and $K$ be a convex $q$-gon  with $q\geq 6$.
 Then $K$ cannot tile $R$.
\end{theorem}

The proof of the above theorem is motivated  by Feng \textit{et al.} \cite{Dej2014}.

In another paper \cite{RRW}, we show that Theorem \ref{thm:six-gon} still holds for $q=5$, with the help of a computer to verify hundreds of cases.
So, for Conjecture 1, the only remaining case is $q=4$, which seems to be very difficult. In this paper, we give a partial answer for this case. A \emph{right-angle trapezoid} is a trapezoid with angles  $\pi/2, \pi/2, \alpha, \pi-\alpha$ where $0<\alpha<\pi/2$, where the two right angles are adjacent.

\begin{theorem}\label{thm:trapezoid}
Let $P$ be a right-angle trapezoid.
 If $P$ can tile a square, then the number of tiles must be even.
\end{theorem}

To prove  Theorem \ref{thm:trapezoid},  we introduce a hypotenuse graph $G$ related to the tiling (\ref{eq:Danzer}).
We show that  Conjecture 1 is true if every connected component of the graph $G$ is Eulerian;
indeed, this is the case when
 $\alpha\neq \pi/3$.
If  $\alpha=\pi/3$, we need to investigate  carefully the forbidden configurations of tiles in  the   tiling (\ref{eq:Danzer}), which is the most difficult part of the proof of Theorem \ref{thm:trapezoid}.

There are some works on other dissection problems of a square into polygons, see for instance
  \cite{ Bor1988, Hag1983, Owi1985}.

\medskip
 The paper is organized as follows. In Section 2, we prove Theorem \ref{thm:six-gon}. In Section 3, we recall some results on  Eulerian graphs.
In Section 4, we define the hypotenuse graph and show that Conjecture 1 is true if the graph is Eulerian.
In Sections 5 and 6, we show that Conjecture 1 is true when $\alpha \neq \pi/3$ and $\alpha=\pi/3$, respectively. In Section 7, we pose several questions.

\section{\textbf{A convex $q$-gon cannot tile a rectangle when $q\geq 6$ }}
Let $\Omega$ be a rectangle   in the  plane. Let $q\geq 3$  and  $P$ be a convex $q$-gon.
Suppose
\begin{equation*}
   \Omega = \cup_{j=1}^N P_j
\end{equation*}
    is a tiling of $\Omega$, where $P_j$, $j\in\{1,\dots, N\}$,  are    congruent to $P$.

 Denote by $\mathcal{V}_\Omega$   the  set consisting of the four vertices of $\Omega$,  and $\mathcal{V}_j$  the vertex set of   $P_j$, $j\in\{1,\dots, N\}$. Let $\mathcal{V}=\cup_{j=1}^N \mathcal{V}_j$.
Clearly ${\mathcal{V}}_\Omega\subset \mathcal{V}$.

For  $w \in \mathcal{V}$, let
$$\mathcal{I}(w):= \{ j: w \in \mathcal{V}_j \};$$
namely, $\mathcal{I}(w)$ is the set of indices  of tiles having $w$ as a vertex.
For $j \in \mathcal{I}(w)$, denote by $\theta_j(w)$ the angle of the vertex $w$ in $P_j$.
Then, for $w \in \mathcal{V}$, we have
\begin{equation}
 \sum_{j \in \mathcal{I}(w)} \theta_j(w) =
     \begin{cases}
        \frac{\pi}{2} & \text{ if } w \in \mathcal{V}_\Omega;\\
        \pi  & \text{ if } w \text{ lies on an (open) side of  a tile or of $\Omega$};\\
        2\pi& \text{ otherwise. }
     \end{cases}
\end{equation}

Define
\begin{equation*}
{\mathcal F} = \{ w \in \mathcal{V}: \sum_{j \in \mathcal{I}(w)} \theta_j(w)=2\pi \}, \ \ {\mathcal H} = \{ w \in \mathcal{V}: \sum_{j \in \mathcal{I}(w)} \theta_j(w)=\pi \}.
\end{equation*}
For a set $A$, let $|A|$ denote its cardinality.

\begin{lemma}
Let
\begin{equation*}
  F=\sum_{w \in \mathcal{F}} |\mathcal{I}(w)|,\  H= \sum_{w \in \mathcal{H}} |\mathcal{I}(w)|,  \ \text{ and } \ \hbar=\sum_{w \in \mathcal{V}_\Omega} |\mathcal{I}(w)|.
\end{equation*}
Then
\begin{equation}
\label{relationofvertices}
\frac{2 |\mathcal{F}| + |\mathcal{H}| + 2}{F+H+\hbar} = \frac{q-2}{q}.
\end{equation}
\end{lemma}

\begin{proof}
Notice that the sum of angles of $P$ is $(q-2)\pi$.
Since
\begin{equation*}
\begin{array}{rl}
\sum_{w \in \mathcal{F}} \sum_{j \in \mathcal{I}(w)} \theta_j(w) &=2 |\mathcal{F}| \pi,\\
\sum_{w \in \mathcal{H}} \sum_{j \in \mathcal{I}(w)} \theta_j(w) &= |\mathcal{H}| \pi,\\
\sum_{w \in \mathcal{V}_\Omega} \sum_{j \in \mathcal{I}(w)} \theta_j(w)&=2\pi,
\end{array}
\end{equation*}
we infer that
\begin{equation*}
  2 |\mathcal{F}| \pi +  |\mathcal{H}| \pi+2\pi=N(q-2)\pi.
\end{equation*}
On the other hand we have
\begin{equation*}
  F+H+\hbar=\sum_{j=1}^N |\mathcal{V}_j| =N\cdot q.
\end{equation*}
Taking the ratio of the above two equations now yields the lemma.
\end{proof}

\begin{lemma}
Let
\begin{equation}
\Delta=F+H+\hbar-3 |\mathcal{F}|-2 |\mathcal{H}|-|\mathcal{V}_\Omega|.
\end{equation}
Then $\Delta \geq 0$ and
\begin{equation}
\label{relation_of_vertices_2}
(q-6) |\mathcal{F}| + (q-4) |\mathcal{H}| + (q-2)\Delta + 2q=8.
\end{equation}
\end{lemma}
\begin{proof}
Since all angles of $P$ are less than $\pi$, we  have $|\mathcal{I}(w)|\geq 3$
   for $w\in \mathcal{F}$, and  $|\mathcal{I}(w)| \geq 2$  for $w\in \mathcal{H}$.
    It follows that $F \geq 3 |\mathcal{F}|, H \geq 2 |\mathcal{H}|$ and $\hbar \geq 4,$
    which imply that $\Delta\geq 0$.

By  \eqref{relationofvertices}  we have
\begin{equation*}
\begin{array}{rl}
2q |\mathcal{F}| + q |\mathcal{H}| + 2q &=(q-2)(F+H+\hbar)\\
              &=(q-2)(3 |\mathcal{F}| + 2 |\mathcal{H}| + 4 + \Delta).
\end{array}
\end{equation*}
Rearranging the terms of the above equation, we obtain \eqref{relation_of_vertices_2}.
The lemma is proved.
\end{proof}

 \medskip
  \noindent \textbf{Proof of Theorem \ref{thm:six-gon}.}
  If $q\geq 6$, then
  the left hand side of  \eqref{relation_of_vertices_2} is no less than $12$, which is absurd.
  $\Box$
  
  \section{\textbf{Eulerian graphs}}

In this section we recall some notions and results of graph theory. See  \cite{Bal2000}.

Let $G=(V, \Gamma)$ be a
directed graph, where $V$ is the \emph{vertex set} and $\Gamma$ is the \emph{edge set}.
Each edge $\gamma$ is associated to an ordered pair $(u,v)$ in $V$,   and we say
$\gamma$ is  \emph{incident out} of $u$ and \emph{incident into} $v$.
 We also call $u$ and $v$ the \emph{origin} and \emph{terminus} of $\gamma$, respectively.
 The number of edges incident out of a vertex $u$ is the \emph{outdegree} of $u$ and is denoted by $\deg^+(u)$. The number of edges incident into a vertex $u$ is the \emph{indegree} of $u$ and is denoted by $\deg^-(u)$.
We remark that in the graph $G$ we allow multi-edges from a vertex $u$ to $v$.

A \emph{directed walk} joining the vertex $v_0$ to the vertex $v_k$ in $G$ is an alternating sequence $v_0 \gamma_1 v_1 \gamma_2 v_2 \dots \gamma_k v_k$ with $\gamma_i$ incident out of $v_{i-1}$ and incident into $v_i$.

Similarly, for an undirected graph, we can define \emph{trail}, \emph{path} and \emph{cycle}, see  \cite{Bal2000}.

$G$ is \emph{connected} if for any $u,v\in V$, there is a path joining $u$ and $v$.
A connected graph $G$ is called \emph{Eulerian} if there is a closed trail containing all the edges of $G$.


  \begin{theorem}\label{thm:tour-directed} Let $G=(V, \Gamma)$ be a connected directed graph. The following three statements are all equivalent:

 $(i)$ For every $u\in V$, $\deg^+(u)=\deg^-(u)$.

 $(ii)$ $G$ is Eulerian.

 $(iii)$  $G$ is a  union of  edge-disjoint directed cycles.
\end{theorem}

\begin{remark} {\rm For an undirected graph, similar results hold, see for instance, \cite{Bal2000};
we will need such results  in Section 6.
}
\end{remark}

To close this section, we give a definition. We say that a (directed or undirected) graph is
\emph{component-wise Eulerian}, if every connected component of the graph is Eulerian.

\section{\textbf{Tiling a square with congruent  right-angle trapezoids}}

Let $\Omega$ be a square in $\mathbb{R}^2$.
   Let $P$ be a trapezoid   with angles $(\alpha, \pi-\alpha, \pi/2, \pi/2)$, where $0< \alpha < \pi/2$; see Figure \ref{fig_tile}. Let
\begin{equation}\label{eq:tiling2}
\Omega=\cup_{j=1}^{N} P_j,
\end{equation}
be a tiling of $\Omega$,
where $P_j$ are all congruent to $P$.
Let $\phi_j$ be the isometry  such that $P_j=\phi_j(P)$, $ j\in\{1,\dots, N\}$.
The rest of the paper proves that $N$ is an even number.

\begin{figure}[!htbp]
\centering
\includegraphics[width=0.4\textwidth]{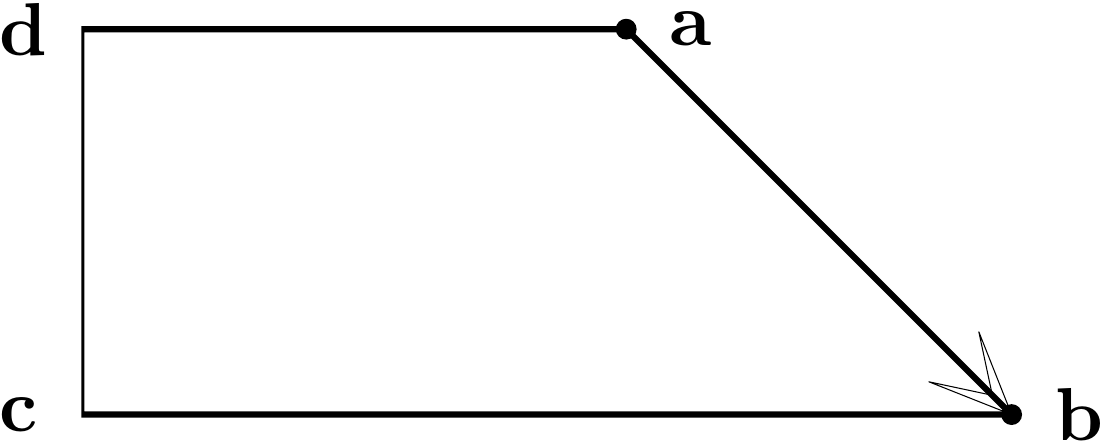}
\caption{The trapezoid $P$.}
\label{fig_tile}
\end{figure}

\subsection{Hypotenuse graph}
We denote the vertices of $P$ by $\mathbf{a},\mathbf{b},\mathbf{c}$ and $\mathbf{d}$; see  Figure \ref{fig_tile}.
The   line segment $\overline{[\mathbf{a},\mathbf{b}]}$ is called the \emph{hypotenuse} of $P$.
We shall define a directed graph consisting of the (directed) hypotenuses of $P_j$, $j\in\{1,\dots, N\}$.
More precisely, let
\begin{equation*}
V=\bigcup_{j=1}^N \{ \phi_j(\mathbf{a}), \phi_j(\mathbf{b})\}
\end{equation*}
be the vertex set.
For each $j\in \{1,\dots, N\}$, we define a directed edge $\tau_j$
with origin $\phi_j(\mathbf{a})$ and terminus $\phi_j(\mathbf{b})$. To emphasize that the edge $\tau_j$ is motivated by 
the hypotenuse of $P_j$, we can even denote $\tau_j$ by the triple
\begin{equation}\label{def:tau}
\tau_j=(\phi_j(\mathbf{a}),\phi_j(\mathbf{b}),P_j)
\end{equation}
Let
\begin{equation*}
 \ \Gamma=\{  \tau_j;~ 1 \leq j \leq N \}
\end{equation*}
 be the set of edges. We call $(V, \Gamma)$ the \emph{hypotenuse graph} of the tiling \eqref{eq:tiling2}.
(It may happen that two different edges have the same origin and terminus, which explains why we put $P_j$ as the third entry of $\tau_j$  in \eqref{def:tau}.)

 The goal of this section is to prove the following:

\begin{theorem}\label{thm:euler} If the hypotenuse graph $(V, \Gamma)$ of the tiling \eqref{eq:tiling2} is component-wise Eulerian,
then $N$ is even.
\end{theorem}

For brevity, we use $\beta$ to denote $\pi-\alpha$ hereafter.

Let $u \in V$.
   If $\theta$ is the  angle of a tile $P_j$ at the vertex $u$, then  we say $\theta$ is an $\textit{angle around}$ $u$.
    If $\theta_1,\dots, \theta_k$ are the  angles around $u$
     arranged in the clock-wise order, then we call $(\theta_1,\dots, \theta_k)$ the \emph{angle pattern} at $u$.

  \begin{lemma}\label{lem:pattern} $(V,\Gamma)$ is component-wise Eulerian if and only if for each $u\in V$, the angle pattern at $u$ falls into
  \begin{equation}\label{eq:angle pattern}
  (\alpha,\beta), \
 (\alpha,\beta,\alpha,\beta), \
 (\alpha,\alpha,\beta,\beta), \   (\alpha,\beta,\pi/2,\pi/2),  \ (\alpha,\pi/2,\beta,\pi/2)
  \end{equation}
  up to a rotation or a reversion.
  \end{lemma}
  \begin{proof} Suppose $(V,\Gamma)$ is component-wise Eulerian.
  For $u\in V$, an angle $\alpha$ at $u$ determines an incoming edge, and an angle $\beta$ at $u$ determines
  an outgoing edge. So the angle pattern at $u$ contains either one $\alpha$ and one $\beta$, or two $\alpha$ and two $\beta$. So  the angle pattern at $u$ falls into \eqref{eq:angle pattern} up to a rotation.

  Assume that all the angle patterns fall into \eqref{eq:angle pattern}. Then $\deg^-(u)=\deg^+(u)$, so $(V,\Gamma)$
  is component-wise Eulerian.
  \end{proof}

\subsection{Pairing principle and feasible cycles}
In the rest of this section, we will always assume that $(V, \Gamma)$ is   component-wise Eulerian.
Let
\begin{equation*}
V_1 = \{ u \in V: \text{the angle pattern at}\  u\  \text{is}\  (\alpha,\alpha,\beta,\beta) \text{ up to a rotation}\}.
\end{equation*}
For each $u \in V_1$,  we denote the tiles around $u$ corresponding to $(\alpha, \alpha, \beta,\beta)$
by
\begin{equation*}
(L_{u,\alpha}, R_{u,\alpha}, R_{u,\beta}, L_{u,\beta}).
\end{equation*}
Then  the $\alpha$ angle of $L_{u, \alpha}$ and the angle $\beta$ in  $L_{u, \beta}$  form an angle measuring $\pi$,
 and so do $\alpha$ in $R_{u, \alpha}$ and  $\beta$ in $R_{u, \beta}$.
 Denote the edges of $\Gamma$ associated to
 $L_{u,\alpha}, R_{u,\alpha}, R_{u,\beta}, L_{u,\beta}$ by   $\ell_{u,\alpha}, r_{u, \alpha}, r_{u,\beta}$ and $\ell_{u, \beta}$,  respectively.  See Figure \ref{fig_aabb}(a).

\begin{figure}[h]
\centering
\subfigure[]
{\includegraphics[width=0.5\textwidth]{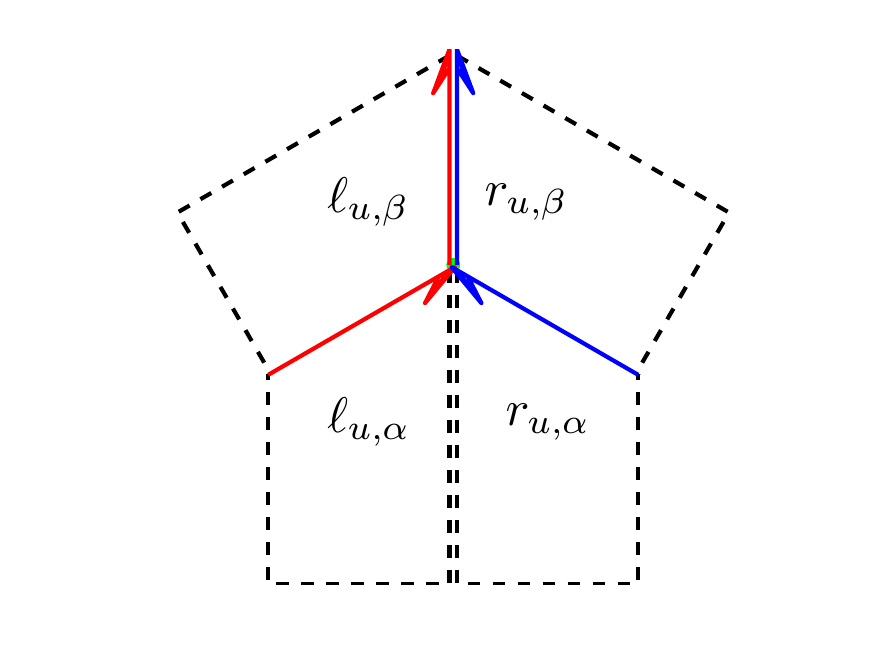}}
\subfigure[]
{\includegraphics[width=0.4\textwidth]{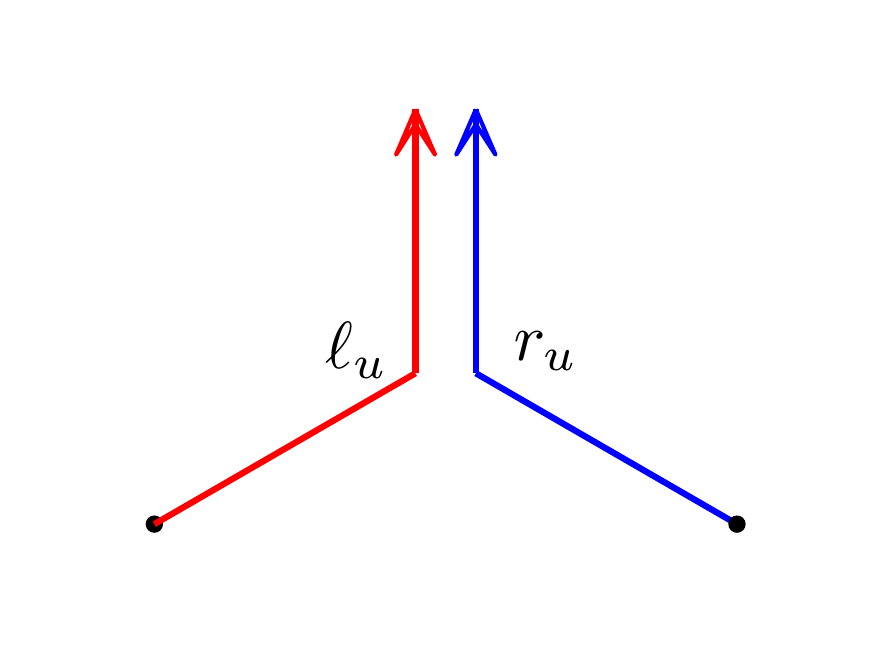}}
\caption{Pairing rule for the angle pattern  $(\alpha,\alpha,
\beta,\beta)$. There are essentially four cases depending on the relative positions of $\alpha$ and $\beta$ in a tile, here we illustrate only one case.}
\label{fig_aabb}
\end{figure}

We regard the path $\ell_{u,\alpha}+\ell_{u,\beta}$ as a single edge, and denote it by
$\ell_u$; similarly, define $r_u=r_{u,\alpha}+r_{u,\beta}$,
see Figure \ref{fig_aabb}(b). Replacing  the old edges by these new edges, we obtain a new graph
\begin{equation}\label{eq:new graph}
\Gamma^\ast = (\Gamma \setminus \cup_{u \in V_1} \{  \ell_{u, \alpha},r_{u, \alpha},\ell_{u, \beta},r_{u, \beta}\}) \cup (\cup_{u \in V_1}\{\ell_u,r_u\}),
\end{equation}
where the corresponding vertex set is $V^\ast=V \setminus V_1$.

A  cycle in $(V,\Gamma)$ is called  \emph{feasible} if it is also a   cycle in $(V^\ast, \Gamma^\ast)$.

Clearly, $(V^*, \Gamma^*)$ is still  component-wise Eulerian since  for each $u\in V^*$,
 the degrees of $u$ in $\Gamma$ and $\Gamma^*$ are the same.
 Therefore, we have

 \begin{lemma}\label{lem:disjoint} $(V,\Gamma)$ is a union of edge-disjoint  feasible cycles.
 \end{lemma}

 \subsection{Structure of feasible cycles} For a sequence of edges $\gamma_1,\dots, \gamma_m$ such that
  the terminus of $\gamma_k$ coincides with the origin of $\gamma_{k-1}$ for all
 $k=1,\dots, m-1$, we use
 $\gamma_1+\cdots+\gamma_m$ to denote the trail formed by these edges.

Let $\mathcal{C}$  be a feasible  cycle in the graph $(V,\Gamma)$ and let us write it as
\begin{equation*}
\mathcal{C}= \gamma_1+\cdots+\gamma_m.
\end{equation*}
 Hereafter, we always use
\begin{equation*}
 K_i=f_i(P)
\end{equation*}
 to denote the tile containing  $\gamma_i$, where $f_i\in \{\phi_1,\dots, \phi_N\}$.
We   denote two  vectors by
\begin{equation*}
\vec{\gamma}_i=\overrightarrow{[f_i(\mathbf{a}), f_i(\mathbf{b})]}  \ \ \text{ and } \ \ \vec{\rho}_i=\overrightarrow{[f_i(\mathbf{d}),f_i(\mathbf{a})]}.
\end{equation*}

 We say $\phi_i(P)$ is \emph{positively oriented} if its vertices
  $\phi_i(\mathbf{a}), \phi_i(\mathbf{b}), \phi_i(\mathbf{c})$ and $\phi_i(\mathbf{d})$ form a clockwise sequence on the boundary of $\phi_i(P)$; otherwise we say $\phi_i(P)$ is
  \emph{negatively oriented}.

 For two edges $\gamma$ and $\gamma'$ in $\Gamma$, we write
\begin{equation*}
   {\gamma}\sim  {\gamma'}
\end{equation*}
   if $\vec{\gamma}$ is either parallel or perpendicular  to $\vec{\gamma'}$.
   Indeed, $\sim$ is an equivalence relation on $\Gamma$.
The following observation is crucial in our discussion.

\begin{lemma}
\label{thm_compatible}
Let $\mathcal{C}=\gamma_1+\cdots+\gamma_m$ be a feasible cycle in $\Gamma$.
 Then for all $i=1,\dots, m$, by identifying  $K_{m+1}$ with $K_1$, we have that

(i) If $K_i$ and $K_{i+1}$ have different orientations, then $\gamma_i\sim \rho_{i+1}$
and   $\rho_i\sim \gamma_{i+1}$.

(ii) If $K_i$ and $K_{i+1}$ have the same orientation, then $\gamma_i\sim \gamma_{i+1}$, and
$\rho_i \sim \rho_{i+1}$.
\end{lemma}

\begin{proof}  Let $v$ be the terminus of $\gamma_i$ as well as  the origin of $\gamma_{i+1}$.
 By Lemma \ref{lem:pattern}, the angle pattern at $v$ must be one of
 \begin{equation}\label{pattern-one}
 (\alpha,\beta), \
 (\alpha,\beta,\alpha,\beta), \  (\alpha,\beta,\pi/2,\pi/2), \
 (\alpha,\alpha,\beta,\beta), \   (\alpha,\pi/2,\beta,\pi/2),
 \end{equation}
 up to a rotation or a reversion.
In the first three cases, the angle of $K_i$ at $v$ and the angle of $K_{i+1}$ at $v$  form an angle measuring $\pi$,
see   Figure \ref{thm_proof_1};
in the fourth case, this is also true since $\mathcal C$ is feasible.
In the final case, $K_i$ and $K_{i+1}$ are separated by two right angles, see Figure \ref{thm_proof_2}.

In Figures \ref{thm_proof_1} and \ref{thm_proof_2}, we illustrate
 all the  possible ways to place  $K_{i}$ and $K_{i+1}$, and there are $8$ of them.
  From the figures, one easily sees that the lemma holds.
\end{proof}

\begin{figure}[h]
\centering
 \subfigure[]
 {\includegraphics[width=0.22\textwidth]{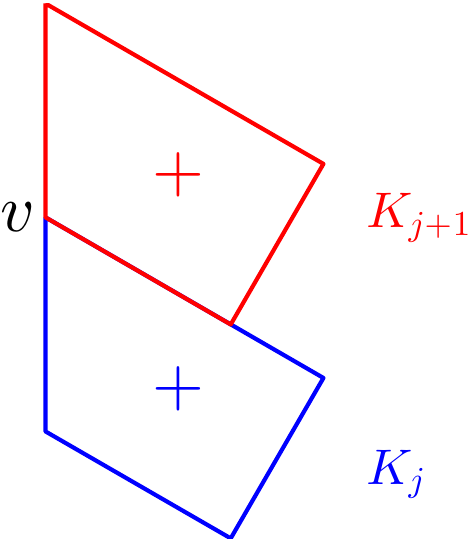}}\hspace{2cm}
\subfigure[]
{\includegraphics[width=0.22\textwidth]{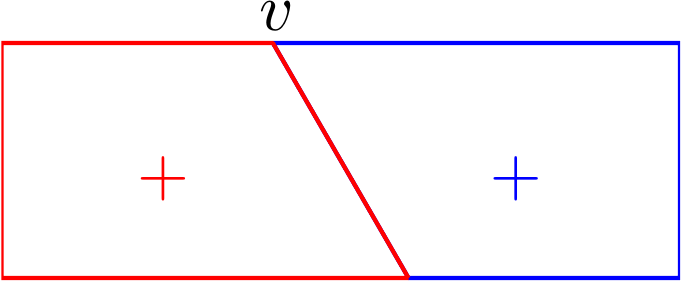}}\\
\subfigure[]
{\includegraphics[width=0.20\textwidth]{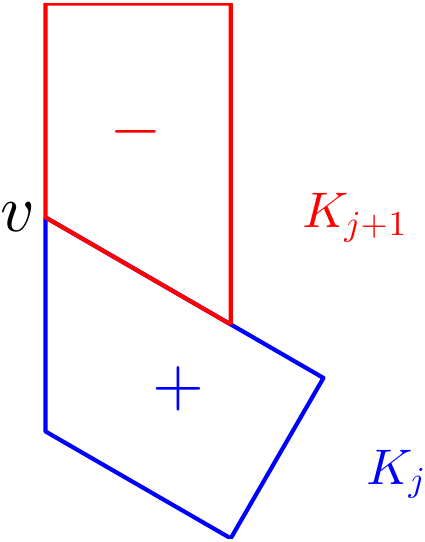}}\hspace{2cm}
\subfigure[]
{\includegraphics[width=0.22\textwidth]{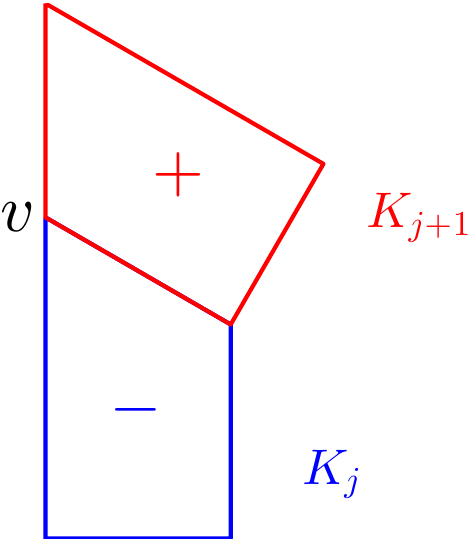}}
\\
\caption{Lemma \ref{thm_compatible}: The four cases that $K_j \cap K_{j+1}$ is a line segment. Here `$+$' means
the tile is positively oriented and `$-$' means the orientation is negative. The corresponding angle pattern is one of the first four cases of \eqref{pattern-one}. }
\label{thm_proof_1}
\end{figure}

\begin{figure}[htbp]
\centering
\subfigure[]
 {\includegraphics[width=0.22\textwidth]{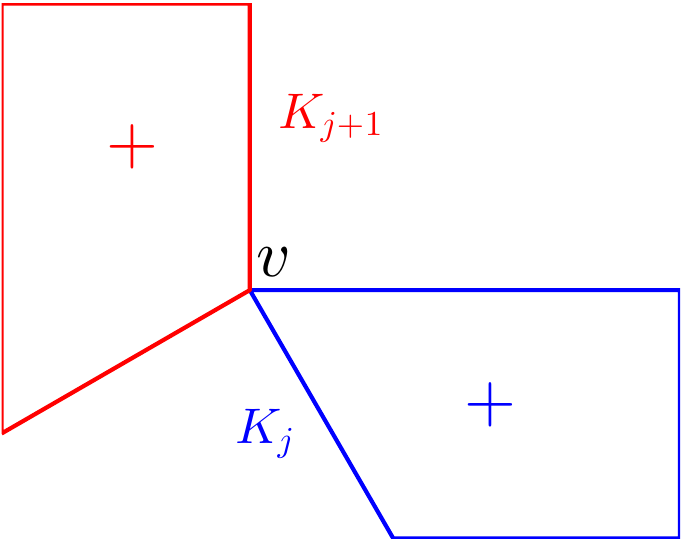}}\hspace{2cm}
\subfigure[]
{\includegraphics[width=0.25\textwidth]{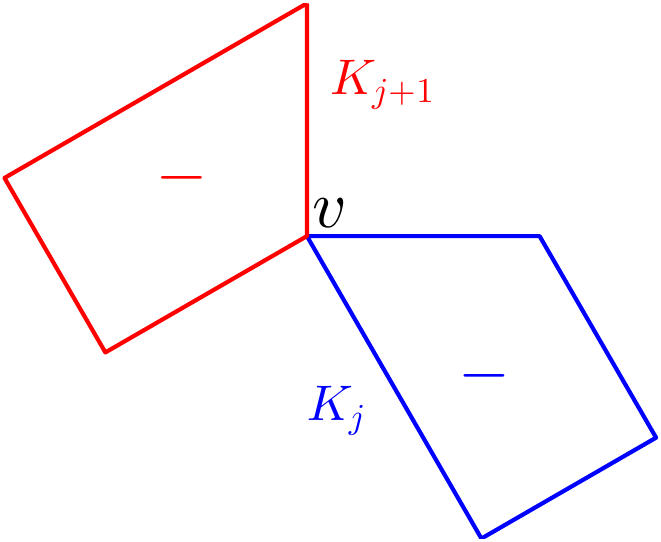}}\\
\subfigure[]
{\includegraphics[width=0.25\textwidth]{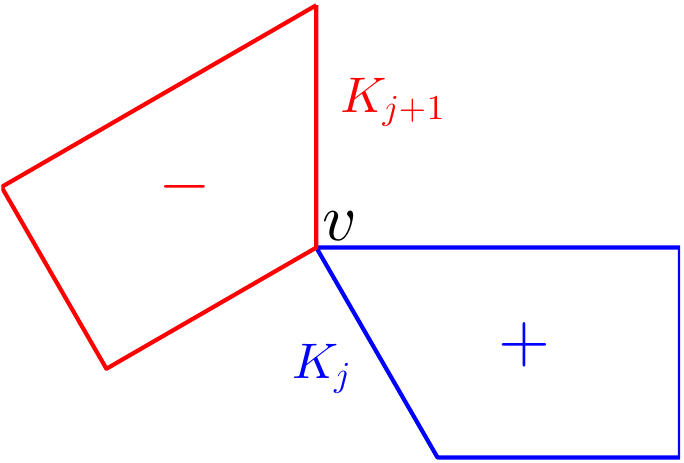}}\hspace{2cm}
\subfigure[]
{\includegraphics[width=0.22\textwidth]{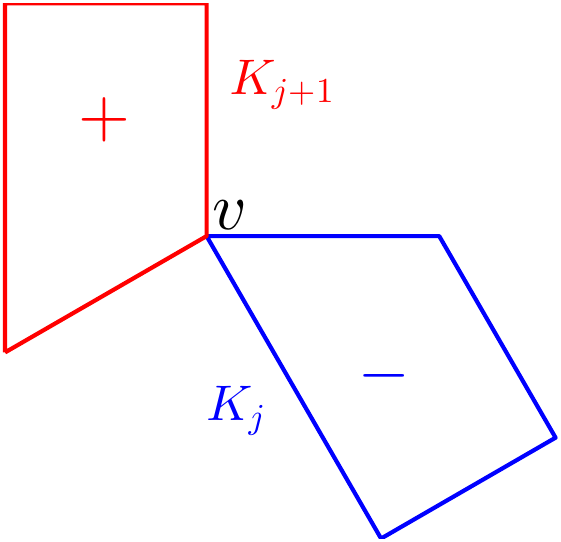}}
\\
\caption{Lemma \ref{thm_compatible}: The four cases that $ K_j \cap K_{j+1}$ is a single  point.
The corresponding angle pattern is the last case of \eqref{pattern-one}.}
\label{thm_proof_2}
\end{figure}

\begin{figure}[htbp]
\centering
\subfigure[Negative orientation]
{\includegraphics[width=0.35\textwidth]{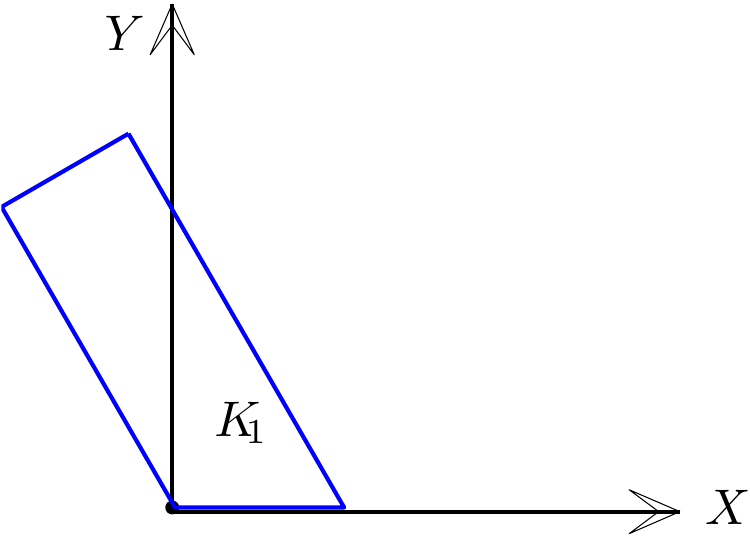}}\hspace{1.5cm}
\subfigure[Positive orientation]
{\includegraphics[width=0.3\textwidth]{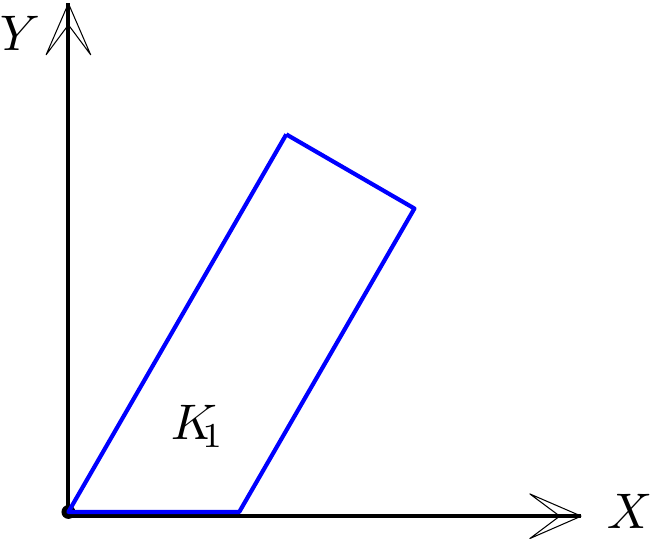}}
\caption{Setting the coordinate system.}
\label{fig_setcoordinate}
\end{figure}

 Let $E$ and $F$ be two points in $\mathbb{R}^2$. We will identify the vector $\overrightarrow{EF}$
to a complex number. We use $\arg z$ to denote the principle argument of a complex number $z$.

\begin{theorem} If $\mathcal{C}=\gamma_1+\cdots+\gamma_m$ is a feasible cycle in $\Gamma$, then $m$ is even.
\label{thm_last}
\end{theorem}

\begin{proof} To facilitate the discussion,
 we set a coordinate system as follows:
If  all   $K_j$ are  negatively oriented, then  we set the coordinate system
 as in Figure \ref{fig_setcoordinate}(a);
  otherwise, we assume $K_1$ has positive orientation without loss of generality, and   set the coordinate system  as in   Figure \ref{fig_setcoordinate}(b).
  In the following we use $A\oplus B$ instead of $A+B$ to emphasize that if $a,a'\in A$ and $b,b'\in B$,
  then $a+b=a'+b'$ holds only when $a=a'$ and $b=b'$.
We claim that
\begin{equation}\label{eq:argument}
\arg \vec\gamma_i \in \{ 0, \pi/2, \pi, 3\pi/2 \}  \oplus \{0, \beta\}.
\end{equation}

If the orientations of $K_1,\cdots,K_m$ are the same,
by Lemma \ref{thm_compatible}, we have $\gamma_j \sim \gamma_1$  for all $j=1,2,\dots, m$. Since $\arg \vec \gamma_1=0$ or $\pi$, we have
$\arg \vec\gamma_i \in \{0, \pi/2, \pi, 3\pi/2 \}.$

If the orientations of $K_1,\cdots,K_m$ are not the same, we claim that

(i) If $K_i$ has negative orientation, then $\arg \vec\gamma_i \equiv \alpha \pmod {\pi/2},$ $\arg \vec\rho_i \equiv 0 \pmod {\pi/2}$;

(ii) If $K_i$ has positive orientation, then $\arg \vec\gamma_i \equiv 0 \pmod {\pi/2},$ $ \arg \vec\rho_i \equiv \alpha \pmod {\pi/2}$.

For $i=1$,  $K_1$ has positive orientation by our convention, and  $\arg \vec \gamma_1 =\pi$ and $\arg \vec\rho_1 =\alpha+\pi$ by our choice of the coordinate system, so the scenario of item (ii) occurs.
Now the claim can be easily proved by Lemma \ref{thm_compatible}. Our claim is proved.

Let $\omega=e^{i\alpha}$. Recall that $|\gamma|$ denotes the length of $\gamma$.
 By applying a dilation to the tiling, we may assume $|\gamma_i|=1$, then
by the above claim, we have
\begin{equation*}
  \vec \gamma_i \in \{1, -1, i, -i, \omega, -\omega, i\omega, -i\omega \}.
\end{equation*}
  Set
\begin{equation*}
  a=| \{  i;~     \vec \gamma_i=1 \}| - | \{i;   \vec \gamma_i=-1 \}|, \quad
  b=| \{  i;~   \vec \gamma_i=i \}| - | \{ i ;~ \vec \gamma_i=-i \}|,
\end{equation*}
\begin{equation*}
  c=| \{  i;~  \vec \gamma_i=\omega \}| - | \{ i;~  \vec \gamma_i=-\omega \}|, \quad
  d=| \{  i;~ \vec \gamma_i=i\omega \}| - | \{ i ;~ \vec  \gamma_i=-i\omega \}|.
\end{equation*}
$\mathcal{C}$ is  closed implies that $a+bi+c\omega+di\omega=0$, so either $a=b=c=d=0$, or
\begin{equation*}
\left |\frac{a+bi}{c+di}\right |=|-\omega|=1.
\end{equation*}
It follows that $a^2+b^2=c^2+d^2$. Then  $(a+b+c+d)^2$ is even, from which it follows that so is $a+b+c+d$.  Therefore,  the number of edges in ${\mathcal C}$
is even.
\end{proof}

Now, Theorem \ref{thm:euler} is an immediate consequence of Lemma \ref{lem:disjoint}  and Theorem \ref{thm_last}.

\section{\textbf{The proof of Theorem \ref{thm:trapezoid} when $\alpha \neq \pi/3$}}
Let $\Omega=\bigcup_{j=1}^N{P_j}$ be a tiling, where each $P_j$ is congruent with a right angle trapezoid $P$ with an angle $0<\alpha<\pi/2$.
Recall that $(V, \Gamma)$ is the hypotenuse graph of the tiling \eqref{eq:tiling2}.

\subsection{The case $\alpha\not\in \{\pi/4, \pi/3\}$}
Let $u \in V$ and let
    $(\beta_1,\cdots,\beta_k)$ be the angle pattern at $u$.
      Then
   $\beta_1 + \cdots + \beta_k$$=\pi/2, \pi \text{ or }2\pi,$
    and we call $\beta_1 + \cdots + \beta_k$   a
   \emph{$V$-decomposition}  at $u$.
Since $u$ is taken from $V$, at least one angle around $u$ is $\alpha$ or $\beta$.

\medskip

\noindent \textbf{Proof of Theorem \ref{thm:trapezoid}  when $\alpha\not\in \{\pi/3,\pi/4\}$.}
  Suppose the hypotenuse graph $(V, \Gamma)$ is not component-wise Eulerian. Then there exists $u\in V$ such that
$ \deg^+(u) < \deg^-(u)$. Suppose the $V$-decomposition at $u$ is
\begin{equation*}
a\alpha+b\beta+c\pi/2, \quad \text{where $0\leq a<b$ and $c\geq 0$}.
\end{equation*}
From
\begin{equation*}
2\pi\geq a\alpha+b\beta+c\pi/2>a(\alpha+\beta)=a\pi,
\end{equation*}
we conclude that $a<2$.

If $a=1$, then we have $(b-1)\beta+c\pi/2=0$ or $\pi$, which is impossible.

If $a=0$, then $b\beta+c\pi/2=\pi$ or $2\pi$, which implies the $V$-decomposition at $u$ is
either $3\beta=2\pi$ or  $2\beta+\pi/2=2\pi$. In the former case $\alpha=\pi/3$ and in the latter case $\alpha=\pi/4$.

 So     $(V, \Gamma)$  must be component-wise Eulerian,
 and $N$ is even by Theorem \ref{thm:euler}.
$\Box$

\subsection{The case $\alpha=\pi/4$}
In this case, instead of using the  hypotenuse  graph  $(V, \Gamma)$, we will use an undirected graph.
Let $(V, \Gamma_0)$ be an undirected graph, which is obtained by regarding every edge $\gamma\in \Gamma$
as an undirected edge.
Clearly  for every $u \in V$, the degree of  $u$ is even. Consequently, $\Gamma_0$
is component-wise Eulerian, and it is an edge-disjoint union of cycles.

\begin{thm}\label{thm:pi/4-2} Any  cycle of $(V, \Gamma_0)$ consists of an even number of edges.
Consequently, $N=|\Gamma_0|$ is an even number.
\end{thm}

\begin{proof}
Let $\mathcal{C}=\gamma_1+\cdots+\gamma_m$ be a   cycle in $\Gamma$. We choose a direction of the cycle, and regard all the edges involved as a directed edge, and then as a vector, and also as a complex number.
 Clearly,
\begin{equation*}
\arg \vec \gamma_i \in \left \{ \frac{k\pi}{4};~  0\leq k \leq 7 \right \}.
\end{equation*}
Therefore, one can show that $m$ is even by a direct calculation, or  by the same argument as in Theorem \ref{thm_last}.
\end{proof}

Consequently, Theorem  \ref{thm:trapezoid} hols when  $\alpha = \pi/4$.

\section{\textbf{The proof of Theorem \ref{thm:trapezoid} when  $\alpha = \pi/3$}}
 Let $\Omega$ be a square, $P$ be a right angle trapezoid  with an angle $\alpha=\pi/3$. Let
 \begin{equation}\label{tiling-3}
 \Omega=\bigcup_{j=1}^N{P_j}
  \end{equation}
  be a tiling of  $\Omega$, where each $P_j$ is congruent with  $P$.
  From now on, we assume that $N$ is an odd number, and we are going to deduce a contradiction.
  For a polygon $P$, we shall use $\partial P$ to denote its boundary.

 Let $\mathcal V$ be the union of the vertex sets of all $P_j$, and $\Lambda$ be the set consisting of all sides of all $P_j$.
Recall that  $(V, \Gamma)$ is the hypotenuse graph of the  tiling  \eqref{tiling-3}.

\begin{definition}{\rm
 For  $u,v \in \mathcal V$,    we call the line segment $\overline{[u,v]}$  a \emph{basic segment}
    if for any $e\in \Lambda$, either $e\subset \overline{[u,v]}$, or
 $e\cap \overline{[u,v]}$ is either a point or empty.

If  $\overline{[u,v]}$  is a basic segment and it is not a proper subset of any other
 basic segment, then we call  $\overline{[u,v]}$  a \emph{maximal segment}.
 }
\end{definition}

For a basic segment $\overline{[u,v]}$, the line containing the segment divides the plane into two parts. If we assume $u$ as the origin and $v$ as the terminus, then we call the left hand side half plane   the \emph{upper part}, and the other half plane the \emph{lower part}.

Denote by $\partial P_j$ the boundary of $P_j$. Clearly $\bigcup_{j=1}^N \partial P_j$ is a non-overlapping  union of maximal   segments.

By applying a dilation, we may assume   the lengths of the four sides of the tile
$P$ to be
\begin{equation*}
x+1,2 ,x,  \sqrt{3}.
\end{equation*}

\begin{lemma}\label{lem:r-s} There exist  $r,s \in \mathbb{Q}$ with $s>0$ such that
 $x=r+s\sqrt{3}$.
\end{lemma}
\begin{proof}
Let $L_j, 1\leq j\leq 4$, be the four sides of $\Omega$. Clearly $L_j$ are maximal segments.
We identify $L_1$ and $L_3$, and identify $L_2$ with $L_4$, so that $L_1=L_3$
and $L_2=L_4$ have both upper part and lower part.
Let $M$ denote the collection of maximal   segments  of the tiling $\Omega=\bigcup_{j=1}^NP_j$ after this identification.

Let $\overline{[u,v]}$ be a maximal line segment. Let $L$ be the line containing $\overline{[u,v]}$.
 Since $\overline{[u,v]}$ is tiled by some sides of tiles  on the upper part of $L$, there exist $a_1,b_1,c_1,d_1 \in \mathbb{N}$ such that
\begin{equation*}
\left |\overline{[u,v]}\right |=a_1 x + b_1 (x+1)+ c_1  \sqrt{3}+2d_1.
\end{equation*}
A similar relation exists at  the lower  part of $L$.
  Hence there exist $a,b,c,d \in \mathbb{Z}$ such that
\begin{equation}\label{eq:abcd}
a  x + b  (x+1)+ c \sqrt3 + 2d =0.
\end{equation}

If $a+b\neq 0$,  setting $r=-\frac{b+2d}{a+b}, s=-\frac{c}{a+b}$, then $x=r+s\sqrt{3}$ and $r,s \in \mathbb{Q}$. The lemma holds in this case.

If $a+b=0$ for every $\overline{[u,v]}\in M$, then $b+2d+c\sqrt{3}=0$, so   $c=0$.
Let $X_{\overline{[u,v]}}$ be the collection of tiles whose side of length $\sqrt{3}$
is a subset of $\overline{[u,v]}$,  then $|X_{\overline{[u,v]}}|$  is an
even number, since each part of  $L$ contains half of these tiles. Since
\begin{equation*}
\{P_1,\dots, P_N\}=\bigcup_{\overline{[u,v]}\in M} X_{\overline{[u,v]}},
\end{equation*}
is a partition, we conclude that  $N$ is even, which is a contradiction. The first assertion is proved.

To prove the second assertion, we use an area argument.
 Denote the areas of $\Omega$ and $P$  by  $S_\Omega$ and $S_P$, respectively. Obviously $S_\Omega=NS_P$ and
\begin{equation*}
S_P=\frac{1}{2}(2x+1)\sqrt{3}=\frac{2r+1}{2}\sqrt{3}+3s.
\end{equation*}
Let  $\ell$ be the  side  length of $\Omega$. Then  $\ell=A+B\sqrt{3}$
 where $A, B\in \mathbb{Q}$.  So
\begin{equation*}
S_\Omega=A^2+3B^2+2AB\sqrt{3}.
\end{equation*}
Hence $A^2+3B^2=3Ns$, which implies  $s\geq 0$.
Finally,  $s\neq 0$ since  $\ell>0$. The second assertion is proved.
\end{proof}

As a direct consequence of $s>0$, we have

\begin{corollary}\label{lem:pure-2}
  The set $\{ax+b(x+1)+c\sqrt{3}; a, b, c\in {\mathbb N}\}$ contains no positive even numbers.
 Therefore,  if the upper part of a basic segment  is tiled by sides of length $2$ only, then so is the lower part.
\end{corollary}

\begin{lemma}\label{lem:3-beta}
  There is a vertex $v \in V$ such that the angle pattern at $v$ is $(\beta, \beta,\beta)$.
\end{lemma}

\begin{proof}
 Since $N$ is odd, the hypotenuse graph of the  trapezoid tiling   is not component-wise Eulerian.
Therefore, since the total number of angles measuring  $\alpha$ and the total number of angles measuring  $\beta$ are equal,
there exists a vertex $u\in V$ such that $\deg^-(u)>\deg^+(u)$, so in the angle pattern at $u$,
the number of angles measuring $\beta$ is larger than those measuring $\alpha$. This can only happen when
the angle pattern is $(\beta,\beta,\beta)$. (See Figure \ref{fig:3beta}.) The lemma is proved.
\end{proof}

\begin{figure}
  \centering
  \subfigure[]{
  \includegraphics[width=0.26\textwidth]{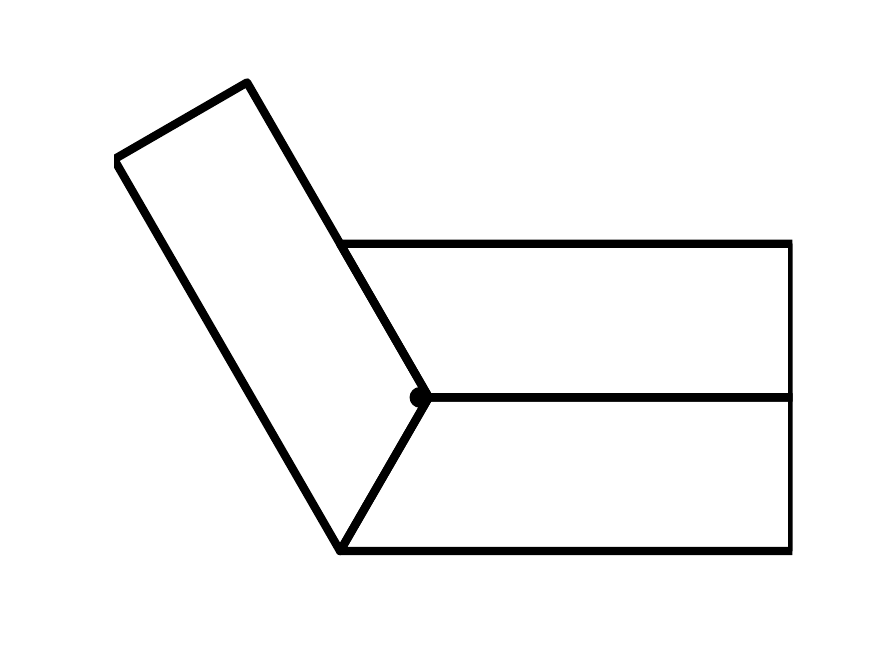}}
  \subfigure[]{
  \includegraphics[width=0.26\textwidth]{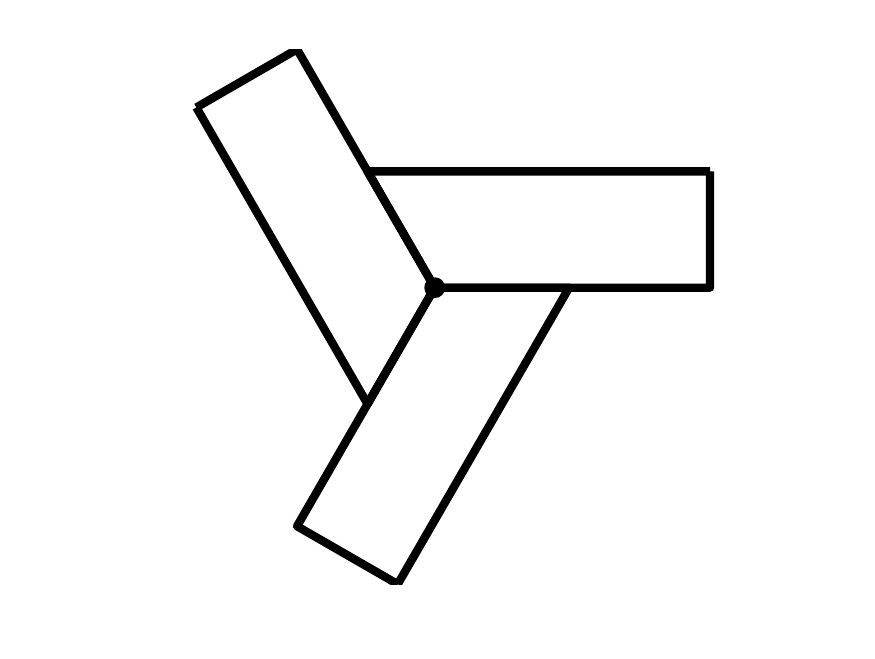}}
   \\
  \caption{Up to symmetry, there are two configurations for the angle pattern $(\beta,\beta,\beta)$.}
  \label{fig:3beta}
\end{figure}

Before proceeding to the proof of Theorem \ref{thm:trapezoid} when $\alpha=\pi/3$, we give some definitions.

\begin{definition} {\rm Let $u,v\in \mathcal V$. We call $\overline{[u, v]}$ a \emph{half maximal segment}
if it is a basic segment and  there exists $u'\in \mathcal V$ such that $\overline{[u',v]}$ is a maximal segment containing $\overline{[u,v]}$.
}
\end{definition}

By definition, a maximal segment itself is a half maximal segment.

Let $\overline{[u,v]}$ be a half maximal segment.
Let $K_1,\dots, K_k$ be the tiles in the upper part of $\overline{[u,v]}$, from left to right, such
that one side of $K_j$ is contained in $\overline{[u,v]}$. We denote the lengths of these sides by $a_j$,
and call $(a_1,\dots, a_k)$ the \emph{upper side sequence} of $\overline{[u,v]}$.
Similarly, we can define the \emph{lower side sequence}.

Let $(a_1,\dots, a_k)$ and $(b_1,\dots, b_{k'})$ be the upper and lower  side sequence of $\overline{[u,v]}$,
respectively. Let $T(j)$ denote the tile contributing the side $a_j$,
and $S(j)$ the tile contributing the side $b_j$.
But for clarity, we will use $T(j,a_j)$ and $S(j,b_j)$ instead of $T(j)$ and $S(j)$.
We call $T(j, a_j)$ an \emph{upper tile}, and $S(j, b_j)$
a \emph{lower tile}.

 \begin{definition} {\rm Let  $\overline{[u,v]}$ be a half maximal segment with
 upper and lower side sequences $(a_1,\dots, a_k)$ and $(b_1,\dots, b_{k'})$, respectively.
 If $a_1=2, b_1\neq 2$, or $a_1\neq 2, b_1=2$,  then we call $\overline{[u,v]}$ a \emph{special segment}.
 }
\end{definition}

By Corollary \ref{lem:pure-2}, we see that if $\overline{[u,v]}$ is a special segment, then neither $(a_1,\dots, a_k)$ nor $(b_1,\dots, b_{k'})$
 is $(2,\dots, 2)$.

Now we regard the points in $V$ as complex numbers.
We define the \emph{head information} of  a special segment $\overline{[u,v]}$ to be
\begin{equation*}
(u, \mathbf{x}, \delta, \theta),
\end{equation*}
where $\mathbf{x}=\frac{v-u}{|v-u|}$; $\delta=upper$ and $\theta$ is the angle of $T(1, a_1)$ at $u$ if $a_1=2$, and
  $\delta=lower$ and $\theta$ is the angle of $S(1, b_1)$ at $u$ if $b_1=2$.

 Let $\omega=\exp(2\pi\mathbf{i}/3)$.  For  a  given vector $\mathbf{x}\neq 0$, we define a partial order on $\mathbb C$ as follows:
  We say $u\prec v$ if   $u\neq v$ and $v-u=a\mathbf{x}+b\omega\mathbf{x}$ with $a,b\geq 0$.
  
  \begin{figure}[htbp]
\centering
  \includegraphics[width=0.5\textwidth]{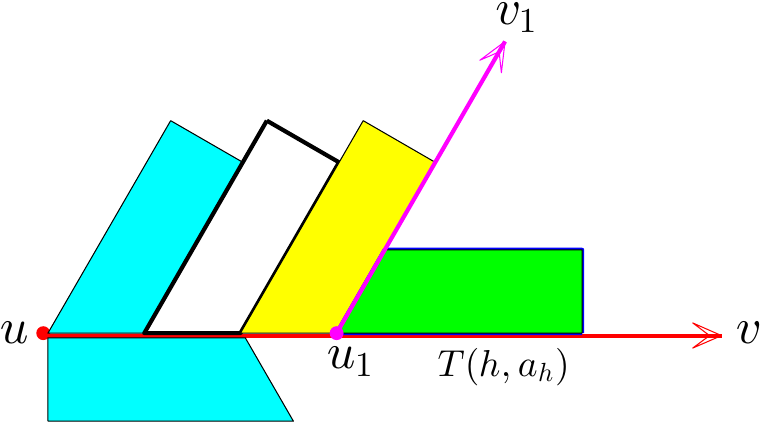}\\
  \caption{$\overline{[u,v]}$ is a special segment with head information  $(u, \mathbf{x}, upper, \alpha)$.
   It produces a new special segment $\overline{[u_1,v_1]}$, where $u_1=u+2(h-1)\mathbf{x}$.  }
  \label{fig:new-1}
\end{figure}

 \begin{lemma}\label{lem:type-II} If $\overline{[u,v]}$ is a special segment with  head information
   $(u, \mathbf{x}, upper, \alpha)$, then there exists a
 special segment  $\overline{[u_1,v_1]}$  with  head information $(u_1, \omega\mathbf{x}, lower, \alpha)$  and $u\prec u_1$.

 Similarly, if $\overline{[u,v]}$ is a special segment with head information
   $(u, \omega\mathbf{x}, lower, \alpha)$, then there exists a
 special segment  $\overline{[u_1,v_1]}$  with  head information  $(u_1, \mathbf{x}, upper, \alpha)$  and $u\prec u_1$.
 \end{lemma}

 \begin{proof} Let $(a_1,\dots, a_k)$ and $(b_1,\dots, b_{k'})$ be the upper and lower side sequence
 of $\overline{[u,v]}$, respectively.

 Suppose the head information   of $\overline{[u,v]}$  is
 $(u, \mathbf{x}, upper, \alpha)$.
Then $a_1=2$ and $b_1\neq 2$. Let $h$ be the minimal integer such that $a_h\neq 2$, then $h\geq 2$.
(The existence of $h$ is guaranteed by Corollary \ref{lem:pure-2}.)
 Let
\begin{equation*}
 u_1=u+ 2(h-1)\mathbf{x}.
\end{equation*}
 By Corollary \ref{lem:pure-2}, $u_1\prec v$.
  Since $T(h-1,a_{h-1})$ contributes an angle $\beta$ at $u_1$,  the tile  $T(h,a_h)$
  must contribute an angle $\alpha$ at $u$, and the orientation of $T(h, a_h)$
 is positive. See Figure \ref{fig:new-1}.  Hence the pattern of Figure \ref{fig:forbidden} (a)  occurs, and there is a special segment
 with head information
 $(u_1, \omega\mathbf{x}, lower, \alpha).$ The first assertion is proved.

 The second assertion can be proved in the same manner as the first one.
 \end{proof}

\begin{corollary}\label{cor:type-II} Special segments with head  information  $(u, \mathbf{x}, \delta, \alpha)$  do not exist.
\end{corollary}

\begin{proof} If $\overline{[u,v]}$ is a special segment with head information $(u,\mathbf{x}, \delta, \alpha)$,
then by Lemma \ref{lem:type-II}, there exists a sequence of special segments $\overline{[u_k, v_k]}$, $k\geq 1$, such that $u_k\prec u_{k+1}$ for all $k$. This implies that $\mathcal{V}$ is an infinite set since
it contains all $u_k$, which is absurd.
\end{proof}

By Corollary \ref{cor:type-II}, the patterns in Figure \ref{fig:forbidden} cannot occur in the   tiling
 \eqref{tiling-3}.

\begin{figure}
\centering
  \subfigure[]{
  \includegraphics[width=0.26\textwidth]{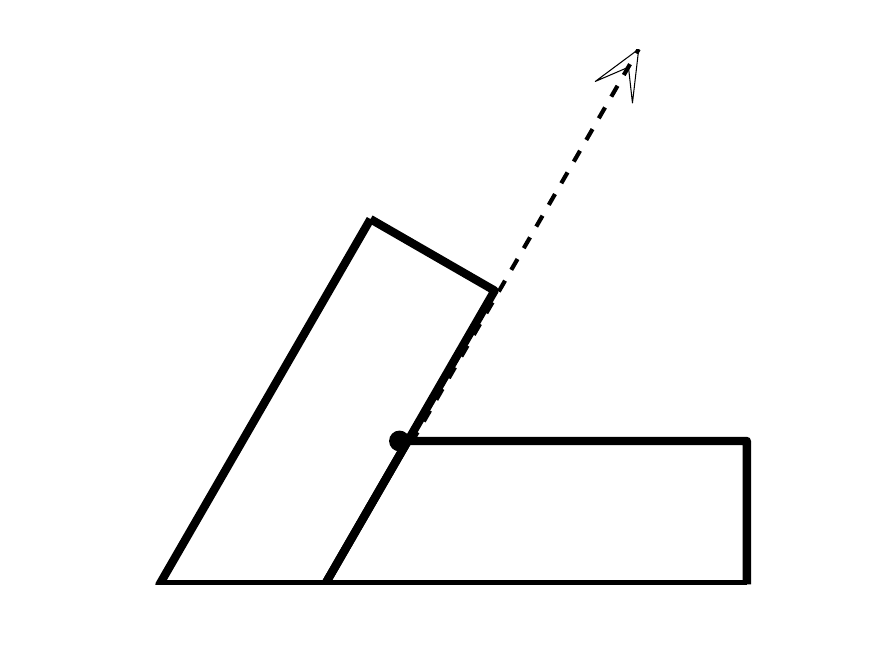}}
  \subfigure[]{
  \includegraphics[width=0.26\textwidth]{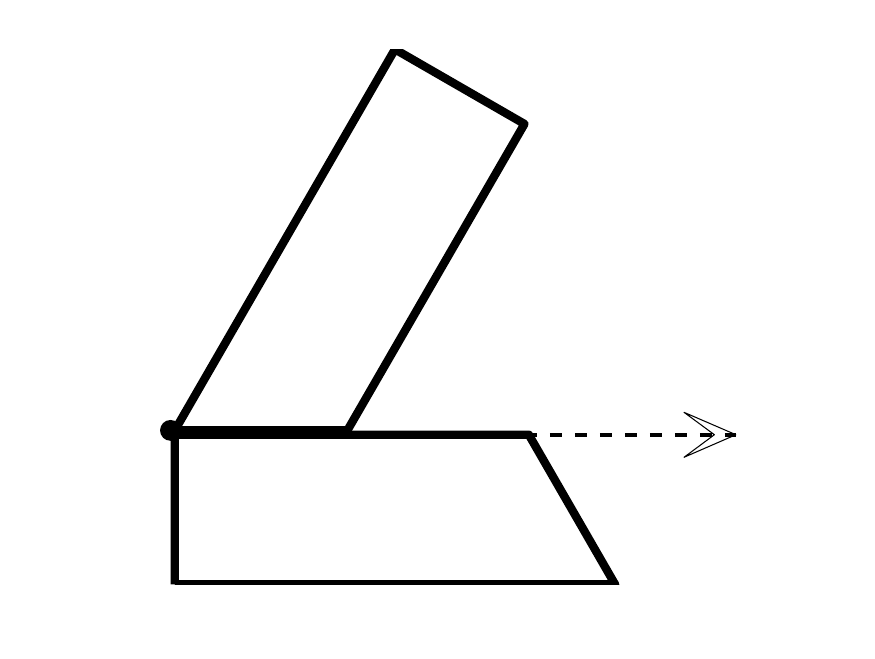}}
  \subfigure[]{
   \includegraphics[width=0.26\textwidth]{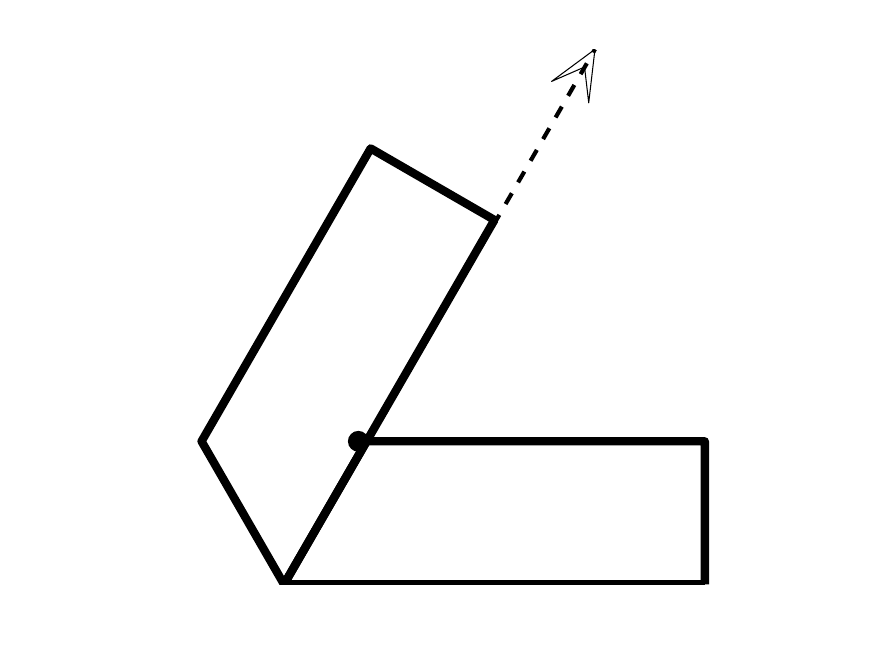}}\\
  \caption{Examples of forbidden patterns. There are many others.}
  \label{fig:forbidden}
\end{figure}

 \begin{lemma}\label{lem:type-I} Let $\overline{[u,v]}$ be a special segment with  head information  in  one of the following forms:
  \begin{equation}\label{eq:info}
  (u,\mathbf{x},upper, \beta), \   (u, \omega \mathbf{x}, lower, \beta).
  \end{equation}
 Then there exists a
 special segment  $\overline{[u',v']}$  with  head information in \eqref{eq:info} (after replacing $u$ by $u'$) and $u\prec u'$.
 \end{lemma}

 \begin{proof}
 Let $(a_1,\dots, a_k)$ and $(b_1,\dots, b_{k'})$ be the upper and lower side sequences of $\overline{[u,v]}$, respectively.

First, let us assume that the  head information  of $\overline{[u,v]}$  is
 $(u,\mathbf{x}, upper, \beta)$. Then $a_1=2, b_1\neq 2$.
 Let $h$ be the minimal integer such that $a_h\neq 2$, then $h\geq 2$.
 (Again the existence of $h$ is guaranteed by Corollary  \ref{lem:pure-2}.)
 Let
\begin{equation*}
 u_1=u+ 2(h-1)\mathbf{x}.
\end{equation*}
 (We remark that $u_1$ is a kind of turning point.)
  Notice that $u\prec u_1$.

 First, we argue that $T(h-1, a_{h-1})$ provides an angle $\alpha$ at $u_1$.
 Otherwise,  $T(h-1, a_{h-1})$ provides an angle $\beta$ at $u_1$ implies that  $T(h, a_h)$  provides an angle $\alpha$
 at $u_1$.
 Then  the tiles $T(h-1, a_{h-1})$ and $T(h, a_h)$  form the forbidden pattern in Figure \ref{fig:forbidden} (a).
Our assertion is proved. Now $a_h=x$ or $x+1$ since $a_h\neq \sqrt{3}$.

\begin{figure}[htbp]
\centering
  \includegraphics[width=0.5\textwidth]{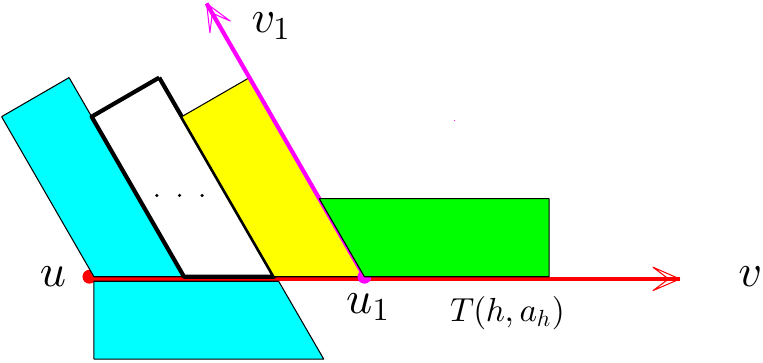}\\
  \caption{Case 1 of the proof of Lemma \ref{lem:type-I}. The yellow tile and the green tile form
  a  special segment with head information $(u_1, \omega\mathbf{x}, lower, \beta)$.
  }
  \label{fig:type-I-2}
\end{figure}

 \emph{Case 1.} If  $a_h=x$, then $T(h, a_h)$ provides an angle $\beta$
 at $u_1$, and there is a half maximal segment $\overline{[u_1,v_1]}$ with direction
 $\mathbf{y}=\omega\mathbf{x}.$
 Moreover, it follows that $\overline{[u_1,v_1]}$ is a special segment with head information
 $(u_1, \omega\mathbf{x}, lower, \beta)$, as we desired.  See Figure \ref{fig:type-I-2}.

 \emph{Case 2.} If $a_{h}=x+1$, then $T(h, a_h)$ provides an angle $\alpha$
 at $u_1$. Let $\overline{[u_1,v_1]}$ be the half maximal segment  with direction $\omega\mathbf{x}.$

 Let $(c_1,\dots, c_q)$ and $(d_1,\dots, d_{q'})$ be the upper and lower side sequence
 of $\overline{[u_1,v_1]}$, respectively. Then $d_1=2$.
We assert that $c_1=2$, for otherwise, the forbidden pattern in Figure \ref{fig:forbidden} (c) will occur.
Let $p$ be the maximal integer such that
 $c_1=\cdots=c_p=d_1=\cdots=d_p=2.$
 Denote
\begin{equation*}
 u_2=u_1+2p(\omega \mathbf{x}).
\end{equation*}

 \emph{Case 2.1.} If $p<q$, then at least one of $c_{p+1}$ and $d_{p+1}$ is not $2$.

 If $c_{p+1}\neq 2$, since  $\overline{[u_1,u_2]}$ is not a half maximal segment,
 the angle pattern at  $u_2$ must be $(\alpha, \alpha, \beta,\beta)$. Let $T$ be the upper tile
 at $u_2$ contributing an angle $\alpha$, then  $T$ must be negatively oriented, so $T$ and the last  upper tile
  $T(p, c_p)$ form a forbidden pattern. (See Figure \ref{fig:type-I-3} (a).)

 If $d_{p+1}\neq 2$, similarly, we get a forbidden pattern at the lower part of $\overline{[u_1,v_1]}$.

 \emph{Case 2.2.} If $p=q$, then every $c_j$ and $d_j$ is $2$, and $p=q=q'$.
 So, $v_1=u_2$, and the angle pattern at $v_1$ must be $(\beta,\beta,\beta)$,
 and the configuration in Figure \ref{fig:3beta} (a) or its reflection occurs.
 Therefore,
 there exists a   special segment  $\overline{[u_2,v_2]}$  with head information
 $(u_2, \omega^2 \mathbf{x}, lower, \beta)$ or $(u_2,   \mathbf{x}, upper, \beta)$,
 see Figure \ref{fig:type-I-3} (b).
 
  \begin{figure}[htbp]
 \centering
  \subfigure[]
  {\includegraphics[width=0.4\textwidth]{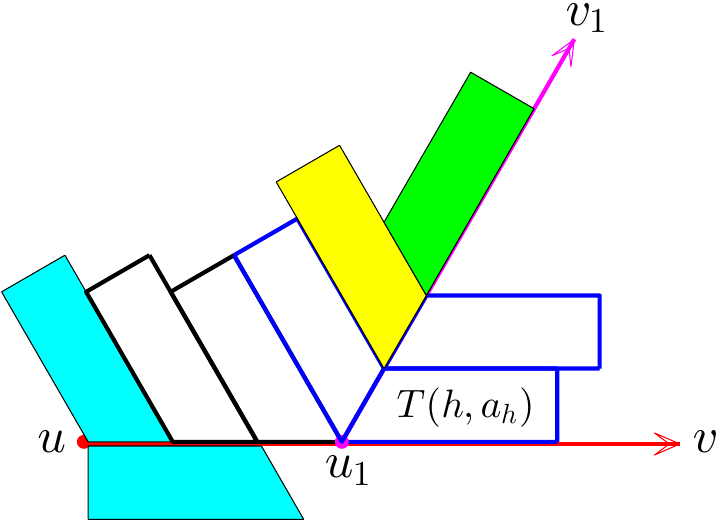}}
  \subfigure[]
  {\includegraphics[width=0.4\textwidth]{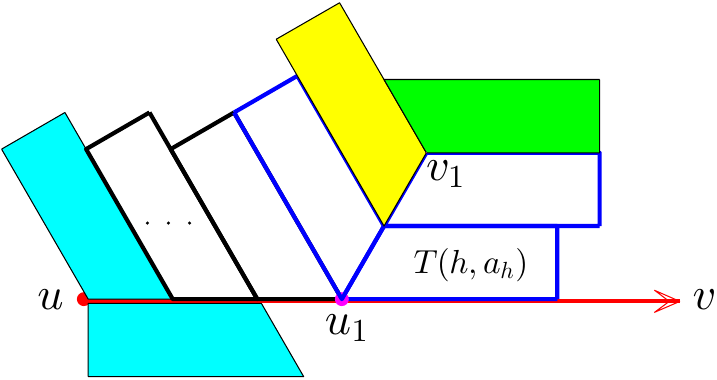}}\\
  \caption{Case 2 of the proof of Lemma \ref{lem:type-I}. (a) Case 2.1: The yellow tile and the green tile form
  a forbidden pattern. (b) Case 2.2: The yellow tile and the green tile
  produce a new special segment.}
  \label{fig:type-I-3}
\end{figure}

 The case that  the head information  of $\overline{[u,v]}$  is
 $(u,\omega \mathbf{x}, lower, \beta)$  can be dealt with in the same manner as above.
The lemma is proved.
 \end{proof}

\noindent\textbf{Proof of Theorem \ref{thm:trapezoid} when $\alpha=\pi/3$.}
By Lemma \ref{lem:3-beta}, there exists a  special segment with head information $(u, \mathbf{x}, \delta, \beta)$, see Figure \ref{fig:3beta}.
  Then, by Lemma \ref{lem:type-I},  there exists a sequence of special segments
$\overline{[u_k, v_k]}$, $k\geq 1$,  such that $u_k\prec u_{k+1}$ for all $k$.
But this  contradicts  the fact that $\cal{V}$ is a finite set. Therefore, the assumption that $N$ is odd is wrong.
$\Box$

\section{\textbf{Some questions}}\label{sec:question}
We close this paper with some questions.

\medskip
 \textbf{ Question 1.} \emph{What kind  of quadrilaterals  can tile a square?}
We believe that if a quadrilateral can tile a square, then it
must contain at least two right angles.
See Figure \ref{fig:tiling}.

\medskip
 \textbf{ Question 2.} \emph{Can we replace the square by a rectangle in Conjecture 1?}
It is seen that the answer is yes for $q\neq 4$.
  For Theorem \ref{thm:trapezoid}, this is also true   except the case  that $\alpha=\pi/3$.
(The only place in which we use that $\Omega$ is a square rather than a rectangle is to prove $s>0$ in Lemma \ref{lem:r-s}.)

\medskip
 \textbf{ Question 3.}
 \emph{How does a right-angle trapezoid tile a square?} Let $P$ be  a right-angle trapezoid
  and  \eqref{eq:Danzer} be a tiling of $\Omega$ by $P$. We believe that  every connected component of the hypotenuse graph is a cycle consisting of two edges. In other words, the tiles must be paired by their hypotenuse. See Figure \ref{fig:tiling} (a).

\begin{figure}[htbp]
\centering
\subfigure[]{
  \includegraphics[width=0.35\textwidth]{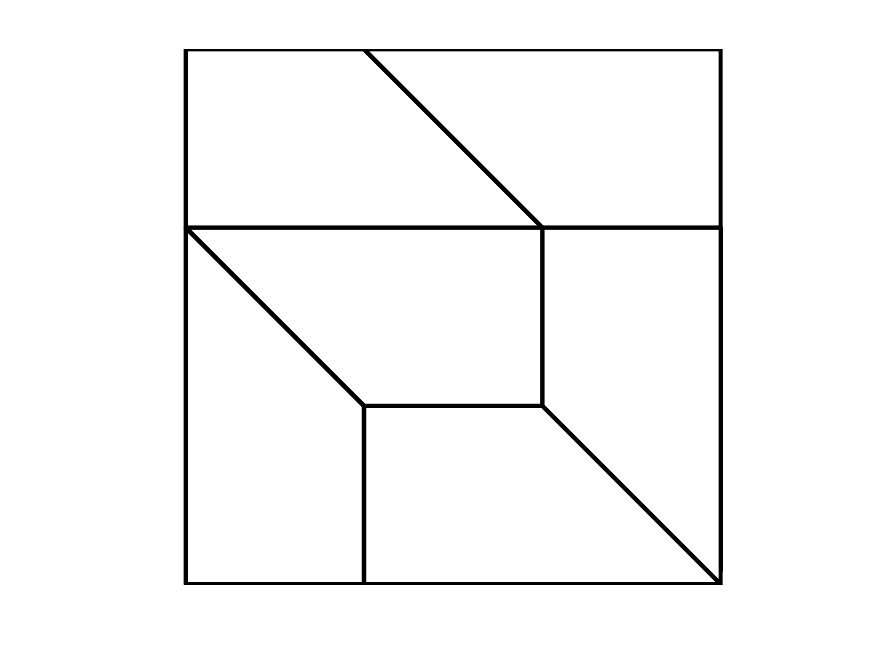}}
  \subfigure[]{
   \includegraphics[width=0.35\textwidth]{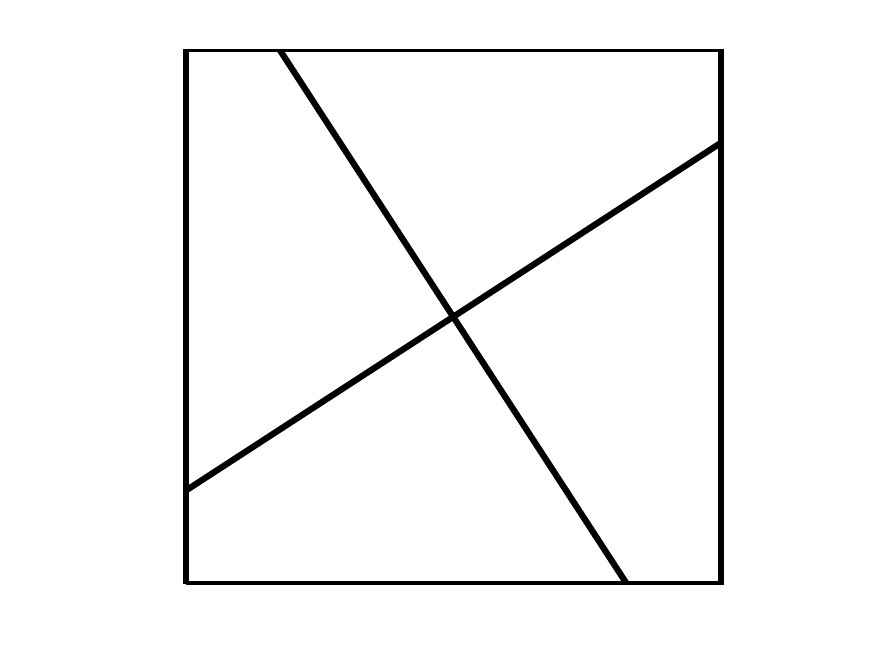}}
   \\
 \caption{Two classes of quadrilaterals which can tile a square.}
 \label{fig:tiling}
\end{figure}

\acknowledgements
We thank the anonymous referees for their careful work and many valuable comments.

\nocite{*}
\bibliographystyle{abbrvnat}
\bibliography{dmtcs_final_ref}
\setcitestyle{numbers}
\label{sec:biblio}

\end{document}